\documentclass[11pt]{amsart}  
\usepackage{bm}
\usepackage{graphicx,verbatim}
\usepackage{amssymb}
\usepackage{amsthm}
\usepackage{amsmath}
\usepackage{bbm}
\usepackage{epstopdf}
\usepackage{textgreek}
\usepackage{amsfonts}
\usepackage{mathrsfs}
\usepackage{lscape} 
\usepackage[vcentermath,enableskew]{youngtab}
\usepackage{tikz,tikz-cd}
\usetikzlibrary{decorations.markings}
\usetikzlibrary{arrows,positioning,automata,shadows,fit,shapes}
\usepackage[english]{babel}
\usepackage{setspace}
\usepackage{microtype}

\usepackage{thmtools}

\usepackage{chngcntr}
\counterwithin{table}{section}

\newtheorem{theorem}{Theorem}[section]
\newtheorem{definition}[theorem]{Definition}
\newtheorem{example}[theorem]{Example}
\newtheorem{corollary}[theorem]{Corollary}
\newtheorem{lemma}[theorem]{Lemma}

\newtheorem{remark}[theorem]{Remark}
\newtheorem*{theorem*}{Theorem}
\newtheorem*{lemma*}{Lemma}
\newtheorem*{definition*}{Definition}

\DeclareMathOperator{\Dim}{dim}

\DeclareMathOperator{\End}{End}
\DeclareMathOperator{\Aut}{Aut}
\DeclareMathOperator{\rank}{rank}

\DeclareMathOperator{\Hom}{Hom} 
\DeclareMathOperator{\Span}{span}
\DeclareMathOperator{\pr}{pr}

\DeclareMathOperator{\Lie}{Lie}

\DeclareMathOperator{\Id}{Id}

\title[Compact Schur-Weyl Duality]{Compact Schur-Weyl duality: real Lie groups and the cyclotomic Brauer algebra}
\author{Kieran Calvert}             
\thanks{Department of Mathematics, University of Manchester, kieran.calvert@manchester.ac.uk}

\begin{document}

\begin{abstract}  
We show that the centraliser of the maximal compact subgroup of the real groups $O(p,q)$ and $Sp_{2n}(\mathbb{R})$ acting on tensors of their standard representation are isomorphic to cyclotomic Brauer algebras. We also show that for $Sp_{2n}(\mathbb{R})$ this cyclotomic Brauer algebra splits into summands of walled Brauer algebras. 

 \end{abstract}

\maketitle                  

\tableofcontents            
 
\begin{section}{Introduction} 

We study real Lie groups $O(p,q)$ and $Sp_{2n}(\mathbb{R})$ and wish to understand their link with Schur-Weyl duality. If one complexifies these groups, we lose the real structure and obtain $O(p+q)(\mathbb{C})$ and $Sp_{2n}(\mathbb{C})$ these complex groups have a well known \cite{B37,We88} Schur-Weyl duality link with the Brauer algebra. There has been vast interest in extensions and variations on this topic, \cite{KX01,RY04,R05,RS06,DDH08,BS12,DRV12,DRV14,ES14}. However, studies in to Schur-Weyl duality and real Lie groups are limited: \cite{CT11,CT12}.

Given the real groups in question are not compact it often makes sense to study their representation theory via a maximal compact subgroup, which we will denote $K$. 
Let $V$ be the standard complex representation of $O(p,q)$ or $S_{2n}(\mathbb{R})$, this paper focuses on the subalgebra $\End_K(V^{\otimes k})$.

We prove the following new examples of Schur-Weyl duality, we believe this is the first occurence of the cyclotomic Brauer algebra as a centraliser algebra.

\begin{theorem} Let $G = Sp_{2n}(\mathbb{R})$ with maximal compact subgroup $K = U(n)$, and $\mathfrak{k} = Lie(K)$. When $k \leq n$ there is a natural isomorphism between $\End_K(V^{\otimes k})$ and the cyclotomic Brauer algebra $Br_{k,2}[-n,0]$ \end{theorem}

\begin{theorem} Let $G = O(p,q)$, $p+q = 2n+1$ with maximal compact subgroup $K = SO(p) \times SO(q)$, and $\mathfrak{k} = Lie(K)$. When $k \leq n$ there is a natural isomorphism between $\End_K(V^{\otimes k})$ and the cyclotomic Brauer algebra $Br_{k,2}[n,p-q]$ \end{theorem} 

Using this isomorphism and the classical double centraliser result we obtain a description of the $K$-types in the polynomial represention of $G$. 

\begin{theorem} For $G = O(p,q)$ or $Sp_{2n}(\mathbb{R})$ the $K$-types arising in the $r^{th}$ polynomial representations are in one-to-one correspondence with irreducible representation of the cyclotomic Brauer algebra $Br_{k,2}[\delta_0,\delta_1]$. \end{theorem}

For the case when $G= Sp_{2n}(\mathbb{R})$ we find a very nice decomposition of the cyclotomic Bruaer algebra into direct sums of matrix algebras tensored with walled Brauer algebras. This gives a Morita equivalence between this cyclotomic Brauer algebra and a direct sum of walled Brauer algebras. 
\begin{theorem} For $\delta_0=-n$ and $\delta_1= 0$ the cyclotomic Brauer algebra associated to the hyperoctahedral group $Br_{k,2}$ is isomorphic to: 

$$Br_{k,2}[-n,0] \cong \bigoplus_{s=0}^kM_{{k\choose s}\times {k \choose s}} WBr_{s, k-s}[-n].$$

\end{theorem}

Unfortunately we weren't able to find a Morita equivalence for the groups $O(p,q)$ but we were able to find an interesting decomposition as a vector space in to Hom spaces, which can be described combinatorially using uneven Brauer diagrams..
\end{section} 

Section \ref{prelim} sets out the basics on Brauer algebras, cyclotomic Brauer algebras and real Lie groups and algebras. Section \ref{injection} proves that the cyclotomic Brauer algebra injects into $\End_K(V^{\otimes k})$ and Section \ref{surjection} proves the same map is a surjection, therefore proving the main results of the paper. 
Section \ref{conseq} looks at the Morita equivalence that is a consequence of the workings in Section \ref{surjection} and asks a few interesting questions about the cyclotomic Brauer algebra.

\begin{subsection}*{Acknowledgements} 


\end{subsection}
\begin{section}{Preliminaries}\label{prelim}

\begin{subsection}{Brauer Algebras}

\begin{definition}{\cite{B37}}\label{brauerpresentation} The rank $k$ Brauer algebra $Br_k[\delta]$, with parameter $\delta\in \mathbbm{k}$, is the associative unital $\mathbbm{k}$-algebra generated by elements $t_{i,i+1}$ and $e_{i,i+1}$ for $i =1,...,k-1$, subject to the conditions:
$$\text{ the subalgebra generated by } t_{i,i+1} \text{ is isomorphic to } \mathbbm{k}[S_k],$$
$$e_{i,i+1}^2 = \delta e_{i,i+1}, $$
$$t_{i,i+1} e_{i,i+1} = e_{i,i+1} t_{i,i+1} = e_{i,i+1},$$
$$t_{i,i+1}t_{i+1,i+2} e_{i,i+1}t_{i+1,i+2}t_{i,i+1} = e_{i+1,i+2},$$
$$[t_{i+1},e_{j,j+1}] = 0 \text{ for } j \neq i,i+1.$$
\end{definition}

The Brauer algebra is usually described as a diagram algebra, below is a map to the diagrams on $2k$ vertices which gives the Brauer algebra as a diagram algebra. Since we will explain the cyclotomic Brauer algebra as a diagram algebra we won't expand anymore. 
One defines an isomorphism on the generators

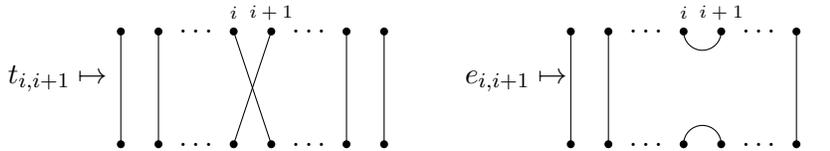
\begin{figure}[ht]
\centerline{
$t_{i,i+1}\mapsto$
\begin{minipage}{45mm}\begin{tikzpicture}[scale=0.5]
  \foreach \x in {0.5,1.5,3.5,4.5,6.5,7.5}
    {\fill (\x,3) circle (3pt);
     \fill (\x,0) circle (3pt);}
     \node[] at (2.5,0) {$\ldots$};
       \node[] at (5.5,0) {$\ldots$};
         \node[] at (2.5,3) {$\ldots$};
           \node[] at (5.5,3) {$\ldots$};
           \node[] at (3.5,3.5) {\tiny$i$};
           \node[] at (4.5,3.5) {\tiny$i+1$};
\begin{scope}
     \draw(0.5,3) -- (0.5,0);
     \draw(1.5,3) -- (1.5,0);
     \draw(4.5,3) -- (3.5,0);
     \draw(3.5,3) -- (4.5,0); 
     \draw(6.5,3) -- (6.5,0); 
     \draw(7.5,3) -- (7.5,0);
  \end{scope}    
\end{tikzpicture}\end{minipage}
$e_{i,i+1}\mapsto$\begin{minipage}{45mm}\begin{tikzpicture}[scale=0.5]
   \foreach \x in {0.5,1.5,3.5,4.5,6.5,7.5}
    {\fill (\x,3) circle (3pt);
     \fill (\x,0) circle (3pt);}
     \node[] at (2.5,0) {$\ldots$};
       \node[] at (5.5,0) {$\ldots$};
         \node[] at (2.5,3) {$\ldots$};
           \node[] at (5.5,3) {$\ldots$};
           \node[] at (3.5,3.5) {\tiny$i$};
           \node[] at (4.5,3.5) {\tiny$i+1$};
\begin{scope}
     \draw(0.5,3) -- (0.5,0);
     \draw(1.5,3) -- (1.5,0);
     \draw(6.5,3) -- (6.5,0); 
     \draw(7.5,3) -- (7.5,0);
     \draw(3.5,3) arc (180:360:0.5 and 0.5);
     \draw(3.5,0) arc (180:360:0.5 and -0.5); 
  \end{scope}    
\end{tikzpicture}\end{minipage}}
 \caption{Isomorphism exhibiting $Br_k(\delta)$ as a diagram algebra}
\label{xyelts}
\end{figure}

We now define the cyclotomic Brauer algebra introduced by H{\'a}ring-Oldenburg \cite{H01}.

\begin{definition} Let $k$ and $m$ be integers. Take any diagram in the Brauer algebra of rank $k$, label every edge by an element in $\{0,1,...,m-1\}$ and give the edge a direction. This is the set of labeled directed Brauer diagrams. \end{definition}
We define an equivalence class on  the set of labeled directed Brauer diagrams. Given a directed edge in the diagram with label $t$, if we change the direction of the edge and replace the label with $-t \mod m$ then we say that the resulting labeled directed Brauer diagram is equivalent to the original.

\begin{example} For $k = 6, m =3$, two examples of labeled directed Brauer diagrams: 

\begin{figure}[ht]
\centerline{
$x=$
\begin{minipage}{45mm}\begin{tikzpicture}[scale=0.5]
  \foreach \x in {0.5,1.5,...,5.5}
    {\fill (\x,3) circle (3pt);
     \fill (\x,0) circle (3pt);}
\begin{scope}[decoration={
    markings,
    mark=at position 0.5 with {\arrow{>}}
    }
    ]
     \draw[postaction={decorate}] (1.5,3) -- (2.5,0); 
    \draw[postaction={decorate}] (2.5,3) arc (180:360:0.5 and 0.5);
     \draw[postaction={decorate}] (0.5,0) arc (180:360:2.5 and -1.2);
  \end{scope}    
\begin{scope}[decoration={
    markings,
    mark= at position 0.2 with{\fill circle (1pt);},
    mark=at position 0.5 with {\arrow{>}}
    }
    ] 
    \draw[postaction={decorate}] (3.5,0) arc (180:360:0.5 and -0.5);    
    \draw[postaction={decorate}] (4.5,3) arc (180:360:0.5 and 0.5);            
  \end{scope}
\begin{scope}[decoration={
    markings,
    mark= at position 0.2 with{\fill circle (1pt);},
    mark=at position 0.5 with {\arrow{>}},
    mark= at position 0.35 with{\fill circle (1pt);}}
    ] 
     \draw[postaction={decorate}]  (0.5,3) -- (1.5,0);    
  \end{scope}
\end{tikzpicture}\end{minipage}
$y=$\begin{minipage}{45mm}\begin{tikzpicture}[scale=0.5]
  \foreach \x in {0.5,1.5,...,5.5}
    {\fill (\x,3) circle (3pt);
     \fill (\x,0) circle (3pt);}
\begin{scope}[decoration={
    markings,
    mark=at position 0.3 with {\arrow{>}}
    }
    ]
     \draw[postaction={decorate}] (1.5,3) -- (0.5,0);
  \end{scope}   
\begin{scope}[decoration={
    markings,
    mark=at position 0.5 with {\arrow{>}}
    }
    ]
     \draw[postaction={decorate}] (2.5,0) arc (180:360:1.5 and -1.2);
  \end{scope}    
\begin{scope}[decoration={
    markings,
    mark= at position 0.2 with{\fill circle (1pt);},
    mark=at position 0.5 with {\arrow{>}}
    }
    ] 
    \draw[postaction={decorate}] (3.5,0) arc (180:360:0.5 and -0.5);  
    \draw[postaction={decorate}] (2.5,3) arc (180:360:1.5 and 1.2);  
    \draw[postaction={decorate}] (3.5,3) arc (180:360:0.5 and 0.5);            
  \end{scope}
\begin{scope}[decoration={
    markings,
    mark= at position 0.2 with{\fill circle (1pt);},
    mark=at position 0.7 with {\arrow{>}},
    mark= at position 0.35 with{\fill circle (1pt);}}
    ] 
     \draw[postaction={decorate}]  (0.5,3) -- (1.5,0);    
  \end{scope}
\end{tikzpicture}\end{minipage}}
 \caption{Two elements in $B_{6,3}(\delta)$}
\label{xyelts2}
\end{figure}
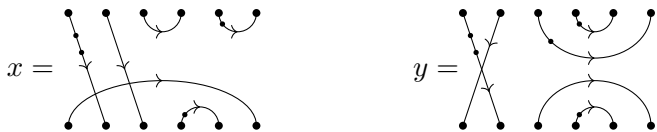

\end{example}
 
 To cocatenate two equivalence classes of labeled directed Brauer diagrams we arrange the directions such that all paths in the concatenated graph go in the same direction, we take the concatenated graph of these two diagrams and then we add the labels on all connected paths in the concatenated diagram. Finally for every closed loop of the diagram with label $i$ we multiply by a predetermined $\delta_i$. Let us highlight this with an example.

\begin{example}\label{multofdiagrams} To multiply $x$ and $y$ we rearrange the directions of $x$ and $y$ to give compatible directions of the concatenated diagram. In this example we need to change the direction of the arc from the the 1st to the 6th bottom vertices in $x$ and the 4th top vertex to the 5th one in $y$. For the arc  in $y$ we change the order and change the label from $1$ to $2 =-1 \mod 3$.

\begin{figure}[ht]
\centerline{
$x\circ y=$
\begin{minipage}{30mm}\begin{tikzpicture}[scale=0.5]
  \foreach \x in {0.5,1.5,...,5.5}
    {\fill (\x,6) circle (3pt);
     \fill (\x,3) circle (3pt);
     \fill (\x,0) circle (3pt);}
\begin{scope}[decoration={
    markings,
    mark=at position 0.3 with {\arrow{>}}
    }
    ]
    \draw[postaction={decorate}] (1.5,3) -- (0.5,0);
  \end{scope}  
\begin{scope}[decoration={
    markings,
    mark=at position 0.5 with {\arrow{>}}
    }
    ]
     \draw[postaction={decorate}] (1.5,6) -- (2.5,3); 
    \draw[postaction={decorate}] (2.5,6) arc (180:360:0.5 and 0.5);
     \draw[postaction={decorate}] (0.5,3) arc (180:360:2.5 and -1.2);
     \draw[postaction={decorate}] (2.5,0) arc (180:360:1.5 and -1.2);
  \end{scope}    
\begin{scope}[decoration={
    markings,
    mark= at position 0.2 with{\fill circle (1pt);},
    mark=at position 0.5 with {\arrow{>}}
    }
    ] 
    \draw[postaction={decorate}] (3.5,3) arc (180:360:0.5 and -0.5);    
    \draw[postaction={decorate}] (4.5,6) arc (180:360:0.5 and 0.5);
    \draw[postaction={decorate}] (3.5,0) arc (180:360:0.5 and -0.5);    
    \draw[postaction={decorate}] (3.5,3) arc (180:360:0.5 and 0.5);
    \draw[postaction={decorate}] (2.5,3) arc (180:360:1.5 and 1.2);              
  \end{scope}
\begin{scope}[decoration={
    markings,
    mark= at position 0.2 with{\fill circle (1pt);},
    mark=at position 0.5 with {\arrow{>}},
    mark= at position 0.35 with{\fill circle (1pt);}}
    ] 
     \draw[postaction={decorate}]  (0.5,6) -- (1.5,3);
     \draw[postaction={decorate}]  (0.5,3) -- (1.5,0);      
  \end{scope}
\end{tikzpicture}\end{minipage}
$=$
\begin{minipage}{30mm}\begin{tikzpicture}[scale=0.5]
  \foreach \x in {0.5,1.5,...,5.5}
    {\fill (\x,6) circle (3pt);
     \fill (\x,3) circle (3pt);
     \fill (\x,0) circle (3pt);}
\begin{scope}[decoration={
    markings,
    mark=at position 0.3 with {\arrow{>}}
    }
    ]
    \draw[postaction={decorate}] (1.5,3) -- (0.5,0);
  \end{scope}  
\begin{scope}[decoration={
    markings,
    mark=at position 0.5 with {\arrow{>}}
    }
    ]
     \draw[postaction={decorate}] (1.5,6) -- (2.5,3); 
    \draw[postaction={decorate}] (2.5,6) arc (180:360:0.5 and 0.5);
     \draw[postaction={decorate}] (5.5,3) arc (180:360:-2.5 and -1.2);
     \draw[postaction={decorate}] (2.5,0) arc (180:360:1.5 and -1.2);
  \end{scope}    
\begin{scope}[decoration={
    markings,
    mark= at position 0.2 with{\fill circle (1pt);},
    mark=at position 0.5 with {\arrow{>}}
    }
    ] 
    \draw[postaction={decorate}] (3.5,3) arc (180:360:0.5 and -0.5);    
    \draw[postaction={decorate}] (4.5,6) arc (180:360:0.5 and 0.5);
    \draw[postaction={decorate}] (3.5,0) arc (180:360:0.5 and -0.5);
       \draw[postaction={decorate}] (2.5,3) arc (180:360:1.5 and 1.2);    
  \end{scope}
\begin{scope}[decoration={
    markings,
    mark= at position 0.2 with{\fill circle (1pt);},
    mark=at position 0.5 with {\arrow{>}},
    mark= at position 0.35 with{\fill circle (1pt);}}
    ] 
     \draw[postaction={decorate}]  (0.5,6) -- (1.5,3);
     \draw[postaction={decorate}]  (0.5,3) -- (1.5,0);      
    \draw[postaction={decorate}] (4.5,3) arc (180:360:-0.5 and 0.5);    
  \end{scope}
\end{tikzpicture}\end{minipage}
$=\delta_0^1$\begin{minipage}{45mm}\begin{tikzpicture}[scale=0.5]
  \foreach \x in {0.5,1.5,...,5.5}
    {\fill (\x,3) circle (3pt);
     \fill (\x,0) circle (3pt);}
\begin{scope}[decoration={
    markings,
    mark=at position 0.3 with {\arrow{>}}
    }
    ]
  \end{scope}   
\begin{scope}[decoration={
    markings,
    mark=at position 0.5 with {\arrow{>}}
    }
    ]
     \draw[postaction={decorate}] (1.5,3) -- (1.5,0);
    \draw[postaction={decorate}] (2.5,3) arc (180:360:0.5 and 0.5);
     \draw[postaction={decorate}] (2.5,0) arc (180:360:1.5 and -1.2);
  \end{scope}    
\begin{scope}[decoration={
    markings,
    mark= at position 0.2 with{\fill circle (1pt);},
    mark=at position 0.5 with {\arrow{>}}
    }
    ] 
    \draw[postaction={decorate}] (3.5,0) arc (180:360:0.5 and -0.5);    
    \draw[postaction={decorate}] (4.5,3) arc (180:360:0.5 and 0.5);            
  \end{scope}
\begin{scope}[decoration={
    markings,
    mark= at position 0.2 with{\fill circle (1pt);},
    mark=at position 0.7 with {\arrow{>}},
    mark= at position 0.35 with{\fill circle (1pt);}}
    ] 
     \draw[postaction={decorate}]  (0.5,3) -- (0.5,0);    
  \end{scope}
\end{tikzpicture}\end{minipage}}
 \caption{Multiplying two elements in $B_{6,3}(\delta)$}
\label{xyelts3}
\end{figure}
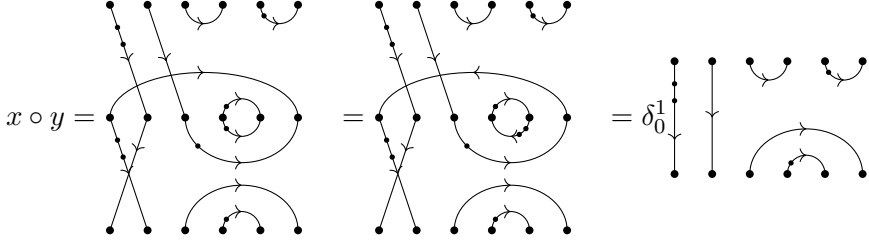

\end{example}
 
\end{subsection}

\begin{definition}\cite{H01}\label{cyclotomicbrauerdef} Let $k$ and $m$ be integers, introduce free variables $\delta= (\delta_0,..,\delta_{m-1})$ such that $\delta_i = \delta_{-i\mod m}$. The cyclotomic Brauer algebra $B_{k,m}(\delta)$ is a free module over $\mathbb{C}[\delta_0,...,\delta_{\lfloor \frac{m}{2}\rfloor}]$ with basis consisting of the equivalence classes of labeled directed Brauer diagrams, multiplication is defined by the multiplication  explained in Example \ref{multofdiagrams}.\end{definition}

We give a presentation for the Cyclotomic Brauer algebra. 

\begin{definition}\label{cyclotomicbrauerpresentation} The cyclotomic Brauer algebra $Br_{k,m}[\delta]$, with parameters $\delta$, is the associative unital $\mathbbm{k}[\delta_0,...,\delta_{\lfloor\frac{m}{2}\rfloor}]$-algebra generated by elements $t_{i,i+1}$ and $e_{i,i+1}$ for $i =1,...,k-1$ and $\theta_j$ for $j=1,...,k$, subject to the conditions:
$$\text{ the subalgebra generated by } t_{i,i+1},\theta_j \text{ is isomorphic to } \mathbb{C}[S_k \rtimes (\mathbb{Z}/m\mathbb{Z})^k],$$
$$\text{ the subalgebra generated by } t_{i,i+1},e_{i,i+1} \text{ is isomorphic to } Br[\delta_0],$$

$$e_{i1,i+1}\theta_i^l e_{i,i+1} = \delta_l e_{i,i+1}$$

$$\theta_i^l \theta_{i+1}^{m-l}  e_{i,i+1} = e_{i,i+1} \theta_i^l \theta_{i+1}^{m-l}= e_{i,i+1},$$
$$[\theta_i,e_{j,j+1}] = 0 \text{ for } j \neq i,i-1.$$
\end{definition}

Throughout this paper the ground field $\mathbbm{k}$ will be fixed as $\mathbb{C}$.

\begin{figure}[ht]
The following figure defines an isomorphism between the two definition we have provided. 
\centerline{
$t_{i,i+1}\mapsto$
\begin{minipage}{45mm}\begin{tikzpicture}[scale=0.5]
  \foreach \x in {0.5,1.5,3.5,4.5,6.5,7.5}
    {\fill (\x,3) circle (3pt);
     \fill (\x,0) circle (3pt);}
     \node[] at (2.5,0) {$\ldots$};
       \node[] at (5.5,0) {$\ldots$};
         \node[] at (2.5,3) {$\ldots$};
           \node[] at (5.5,3) {$\ldots$};
           \node[] at (3.5,3.5) {\tiny$i$};
           \node[] at (4.5,3.5) {\tiny$i+1$};
\begin{scope}[decoration={
    markings,
    mark=at position 0.4 with {\arrow{>}}
    }
    ]
     \draw[postaction={decorate}] (0.5,3) -- (0.5,0);
     \draw[postaction={decorate}] (1.5,3) -- (1.5,0);
     \draw[postaction={decorate}] (4.5,3) -- (3.5,0);
     \draw[postaction={decorate}] (3.5,3) -- (4.5,0); 
     \draw[postaction={decorate}] (6.5,3) -- (6.5,0); 
     \draw[postaction={decorate}] (7.5,3) -- (7.5,0);
  \end{scope}    
\end{tikzpicture}\end{minipage}
$e_{i,i+1}\mapsto$\begin{minipage}{45mm}\begin{tikzpicture}[scale=0.5]
   \foreach \x in {0.5,1.5,3.5,4.5,6.5,7.5}
    {\fill (\x,3) circle (3pt);
     \fill (\x,0) circle (3pt);}
     \node[] at (2.5,0) {$\ldots$};
       \node[] at (5.5,0) {$\ldots$};
         \node[] at (2.5,3) {$\ldots$};
           \node[] at (5.5,3) {$\ldots$};
           \node[] at (3.5,3.5) {\tiny$i$};
           \node[] at (4.5,3.5) {\tiny$i+1$};
\begin{scope}[decoration={
    markings,
    mark=at position 0.5 with {\arrow{>}}
    }
    ]
     \draw[postaction={decorate}] (0.5,3) -- (0.5,0);
     \draw[postaction={decorate}] (1.5,3) -- (1.5,0);
     \draw[postaction={decorate}] (6.5,3) -- (6.5,0); 
     \draw[postaction={decorate}] (7.5,3) -- (7.5,0);
     \draw[postaction={decorate}] (3.5,3) arc (180:360:0.5 and 0.5);
     \draw[postaction={decorate}] (3.5,0) arc (180:360:0.5 and -0.5); 
  \end{scope}    
\end{tikzpicture}\end{minipage}}
\centerline{
$\theta_i\mapsto$
\begin{minipage}{45mm}\begin{tikzpicture}[scale=0.5]
  \foreach \x in {0.5,1.5,3.5,5.5,6.5}
    {\fill (\x,3) circle (3pt);
     \fill (\x,0) circle (3pt);}
     \node[] at (2.5,0) {$\ldots$};
       \node[] at (4.5,0) {$\ldots$};
         \node[] at (2.5,3) {$\ldots$};
           \node[] at (4.5,3) {$\ldots$};
           \node[] at (3.5,3.5) {\tiny$i$};
\begin{scope}[decoration={
    markings,
    mark=at position 0.5 with {\arrow{>}}
    }
    ]
     \draw[postaction={decorate}] (0.5,3) -- (0.5,0);
     \draw[postaction={decorate}] (1.5,3) -- (1.5,0);
     \draw[postaction={decorate}] (5.5,3) -- (5.5,0); 
     \draw[postaction={decorate}] (6.5,3) -- (6.5,0);
  \end{scope}      
\begin{scope}[decoration={
    markings,
    mark= at position 0.2 with{\fill circle (1pt);},
    mark=at position 0.5 with {\arrow{>}}
    }
    ]
     \draw[postaction={decorate}] (3.5,3) -- (3.5,0);   
  \end{scope}
\end{tikzpicture}\end{minipage}}
 \caption{Isomorphism exhibiting $Br_{k,m}(\delta)$ as a diagram algebra}
\end{figure}
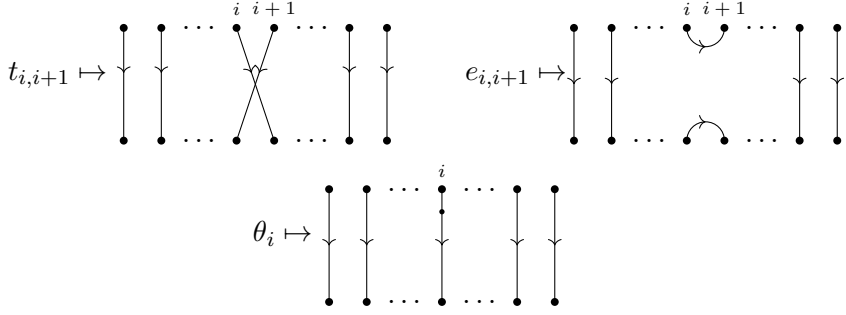

\begin{subsection}{Real Lie Groups and Lie algebras}

Throughout this paper we fix the following notation.
Let $G$ be $O(p,q)$, $p+q=2n+1$ or $Sp_{2n}(\mathbb{R})$. Let $\mathfrak{g}_0$ be $Lie(G)$, with complexification $\mathfrak{g}= \mathfrak{g}_0\otimes_\mathbb{R} \mathbb{C}$. We uniformly denote a real Lie algebra by $\mathfrak{g}_0$, for a complex Lie algebra we drop the subscript. We fix a Cartan involution $\theta$ of $\mathfrak{g}_0$ and extend to $\mathfrak{g}$, let $\Theta$ be the corresponding involution of $G$. A maximal compact subgroup of $G$ is $K$, the fixed space of $\Theta$. The Lie algebra $\mathfrak{g}_0$ decomposes as $\mathfrak{k}_0 \oplus \mathfrak{p}_0$. The subspace  $\mathfrak{p}_0$ is the $-1$ eigenspace of $\theta$, the subalgebra $\mathfrak{k}_0$ is the $+1$ eigenspace of $\theta$ and the Lie algebra of $K$. Similarly, $\mathfrak{g} = \mathfrak{k} \oplus \mathfrak{p}$.  Let $\mathfrak{a}_0$ be a maximal commutative Lie subalgebra of $\mathfrak{p}_0$. 
Let $M$ be the centralizer of $\mathfrak{a}_0$ in $K$ under the adjoint action. We have $\Lie(M) =\mathfrak{m}_0$. 

\begin{definition} For $G$ equal to $O(p,q)$ or  $Sp_{2n}(\mathbb{R})$ write $V_0$ for the defining matrix module. That is $\rho : G \to GL(V_0)$ is the injection defining $G$ as a linear group. Write $V = V_0\otimes_\mathbb{R} \mathbb{C}$ for the complexification of $V_0$. 
\end{definition} 
If $G=Sp_{2n}(\mathbb{R})$ then $V=\mathbb{C}^{2n}$ and if $G =  O(p,q)$ then $V = \mathbb{C}^{2n+1}$. 
When $G = Sp_{2n}(\mathbb{R})$, let $e_1,...,e_{2n}$  be the standard matrix basis of $V$, then define a new basis $f_i = e_i + e_{n+i}$ for $i = 1,..,n$ and $f_i' = e_i - e_{n+i}$  for $i = 1,..,n$. We also label $f_i$ by $f_i^1$ and $f_i'$ by $f_i^{-1}$. When $G = O(p,q)$ then $V$ has basis $e_1,...,e_{2n+1}$, we let $f_i = e_{p-i+1} + e_{p+i}$ and $f_i' = e_{p-i+1} - e_{p+i}$.

\begin{definition}\label{casimir} Given a finite dimensional complex Lie algebra $\mathfrak{g}$ with basis $B$ and dual basis $B^*$ with respect to the Killing form, we define the Casimir element in the enveloping algebra $U(\mathfrak{g})$ to be 
$$C^\mathfrak{g} =\sum_{b\in B} bb^* \in U(\mathfrak{g}).$$
For a subalgebra $\mathfrak{h} \subset \mathfrak{g}$ we denote the Casimir element of $\mathfrak{h}$ in $\mathfrak{g}$ by $C^\mathfrak{h} = \sum_{b \in B \cap \mathfrak{h}} b b^*$ where the basis $B$ is taken such that $B \cap \mathfrak{h}$ is a basis of $\mathfrak{h}$.

Similarly we define $\Omega_g \in U(\mathfrak{g}) \otimes U(\mathfrak{g})$ by 
$$\Omega_{\mathfrak{g}} = \sum_{b \in B} b \otimes b^*.$$ 
\end{definition}
\end{subsection}
\end{section} 

\begin{section}{Extending the Brauer algebra by the Cartan involution}\label{injection}

In this section we take the algebra $\End_\mathfrak{g}(V^{\otimes k})$ and relax the commutation condition. We want to study the larger algebra $\End_\mathfrak{k}(V^{\otimes k})$. We extend by the Cartan involution and show that the resulting algebra is the cyclotomic Brauer algebra. In the following section we show that the cyclotomic Brauer algebra is the full centraliser of $\mathfrak{k}$.

\begin{lemma}\label{decompofVtensorV}  If $\mathfrak{g} = \mathfrak{sp}_{2n}$ or $\mathfrak{so}_{2n+1}$ then $V \otimes V$ decomposes as 
$$\Lambda^2 V \oplus S^2V/\mathbbm{1} \oplus \mathbbm{1} \text{ for } \mathfrak{so}_{2n+1},$$
$$\Lambda^2 V / \mathbbm{1} \oplus S^2 V \oplus \mathbbm{1} \text{ for } \mathfrak{sp}_{2n}.$$
Here $\mathbbm{1}$ denotes the trivial module of $\mathfrak{g}$.
For $SO_{2n+1}$ the trivial module $\mathbbm{1}$ is the span of $v_\mathbbm{1} =\sum_{i=1}^n{e_i \otimes e_i}$ for $Sp_{2n}(\mathbb{R})$then $\mathbbm{1}$ is the span of $v_\mathbbm{1}=\sum_{i=1}^n f_i \otimes f_i' - f_i' \otimes f_i$.
\end{lemma}
Let $s_{12}$ be the operator that has $1$-eigenspace $S^2V$ and $-1$ eigenspace $\Lambda^2 V$, alternatively $s_{12}$ swaps the tensor order in $V \otimes V$. Similar let $s_{i,i+1}$ swap the $i^{th}$ and ${i+1}^{st}$ tensors in $V^{\otimes k}$.
Let $\pr_1$ be the projection of $V\otimes V$ onto the trivial submodule $\mathbbm{1}$ in the decomposition above. Let $\pr_{i,i+1}$ be the projection onto the trivial submodule of $V_i \otimes V_{i+1}$. 

\begin{theorem}{\cite[Theorem 46]{B37}}\label{Brauarcentraliser} Let $\rho$ denote the action of $SO_{2n+1}(\mathbb{C})$ (resp. $Sp_{2n}(\mathbb{C})$) on $V^{\otimes k}$. Define $\pi$ from $Br_k[n]$ (resp. $Br_k[-n]$) to $\End V^{\otimes k}$ by:
$$\pi(t_{i,i+1}) = s_{i,i+1} \in \End(V^{\otimes k}),$$
$$\pi(e_{i,i+1}) =  n pr_{i,i+1} (\text{resp. } -n pr_{i,i+1}) \in \End(V^{\otimes k}$$.

We have the following equalities: 

$$\begin{array}{lcr}
\rho(G) & = & \End_{Br_k} (V^{\otimes r}),\\
\pi(Br_k) & = & \End_{G} (V^{\otimes r}).\\
\end{array} 
$$
Furthermore if $k \leq n$ then the representation $\pi$ is faithful. \end{theorem} 

The above theorem is Brauer's celebrated result. We use this as a base and extend the Brauer algebra to obtain a larger algebra that commutes with a maximal compact subalgebra $K$ of $G$ or equivalently $\mathfrak{k}\subset \mathfrak{g}$.

\begin{definition} Given the action of $U(\mathfrak{g})^{\otimes k}$ of $\End(V)^{\otimes k}$ on $ V^{\otimes k}$ we write $(X)_i$ for the action of $X$ on the $i^{th}$ tensor in $U(\mathfrak{g})^{\otimes k}$ or $\End(V)^{\otimes k}$,
$$(X)_i = \overbrace{ id \otimes ...\otimes id }^{i -1 \text{ times}} \otimes X \otimes\overbrace{id\otimes ...\otimes id}^{k-i \text{ times}}.$$
\end{definition}

Recall the Cartan decomposition, $\mathfrak{g} = \mathfrak{k} \oplus \mathfrak{p}$ there is a Cartan involution $\theta$ that exhibits this decomposition. For $G$ equal to $Sp_{2n}(\mathbb{R})$ or $SO(p,q)$ there is an involutive element $\xi \in G$ such that $\theta (\mathfrak{g}) = \xi^{-1} \mathfrak{g} \xi$. 
The matrix $\xi$ defines an element in $\Aut(\mathfrak{g})$ and since $\theta$ is the identity on $\mathfrak{k}$ then $\xi^{-1} X - X \xi = 0$ for all $X \in \mathfrak{k}$. The element $\xi$ acts on the standard representation $V$. We define operator $\xi_i$ on $V^{\otimes k}$: 
$$\xi_i = Id^{\otimes i-1}\otimes \xi \otimes Id^{\otimes k-i} \in \End (V^{\otimes k}).$$
\begin{remark}\label{xiaction}
If $G = SO(p,q)$ then $$\xi(e_i) = \begin{cases}e_i & \text{ if } i \leq p,\\ -e_i &\text{ if } i > p.\end{cases}$$ 
If $G=Sp_{2n}(\mathbb{R})$using the basis $\{f_i, f_i'\}$ of $V$ then 
$$\xi(f_i) = \sqrt{-1}f_i', $$ $$\xi(f_i') = -\sqrt{-1}f_i.$$
\end{remark}
\begin{lemma}\label{xicommutes} Let $G$ be $Sp_{2n}(\mathbb{R})$ or $SO(p,q)$ and $(\xi_i) \in \End(V^{\otimes k})$  defined above then the operator $\xi_i$ commutes with the maximal compact subgroup $K$ (resp. $\mathfrak{k})$.
\end{lemma} 

\begin{proof} 

It is sufficient to show that $\xi_i$ commutes with $\mathfrak{k}$. Let $X \in \mathfrak{k}$ then then the action of $X$ on $V^{\otimes k}$ is 
$$\rho^k(X) = \sum_{j=1}^n \rho(X)_j = \sum_{j=1}^n (X)_j,$$
then the commutator of $X$ and $\xi_i$ in $\End(V^{\otimes k})$ is 
$$[X, \xi_i] = [\sum_{j=1}^n (X)_j, \xi_i] = \sum_{j=1}[X_j,\xi_i] = \sum_{j\neq i} [X, Id]_j +[X,\xi] = 0.$$
The last equality follows from the fact that $[X,\xi] = 0$ as shown in Lemma \ref{xicommutes}.
\end{proof}

Since $\xi_i$ commutes with $\mathfrak{k}$ we can extend the Brauer algebra $Br_k \subset \End(V^{\otimes k})$ by the elements $(\xi_i)$ in $\End(V^{\otimes k})$. 

\begin{theorem}\label{cyclbrauerincentraliser} Let $G = Sp_{2n}(\mathbb{R})$or $SO(p,q)$, $p >q$. For $k \leq\rank(G)$ The subalgebra of $\End_{\mathfrak{k}}(V^{\otimes k})$ generated by $Br_k$ and $\{\xi_i : i =1,...,k\}$ is isomorphic to the cyclotomic Brauer algebra $Br_{k,2}[\delta_0,\delta_1]$ with specific $\delta_0$ and $\delta_1$ given by 
$$G = Sp_{2n} \hspace{1cm} \delta_0 = -n, \delta_1 = 0,$$
$$G= SO(p,q) \hspace{1cm} \delta_0 = \lfloor \frac{p+q}{2} \rfloor,\delta_1 = p-q.$$
For $k > \rank(G)$ then there is surjective map from the cyclotomic Brauer algebra to subalgebra generated by $s_{i,i+1}, pr_{i,i+1}$ and $\xi_i$. \end{theorem}  

We prove this theorem in several parts. First we show that a map from $Br_{k,2}$ to $\End_\mathfrak{k}(V^{\otimes k})$ is well defined. Then we prove that this map is injective. 

We define a map from $Br_{k,2}[\delta]$ to $\End_\mathfrak{k}(V^{\otimes k})$. From now on $\delta_0$ and $\delta_1$ will be values corresponding to  Theorem \ref{cyclbrauerincentraliser}.
 \begin{subsection}{Defining a map from $Br_{k,2}$ to $\End_\mathfrak{k}(V^{\otimes k})$}
\begin{definition}\label{cyclbrauerinj} We define a map from $Br_{k,2}$ to $\End(V^{\otimes k})$ by the following definition on generators: 

$$\Phi(t_{i,i+1}) = s_{i,i+1},$$
$$\Phi(e_{i,i+1}) = \delta_0 pr_{i,i+1},$$
$$\Phi(\theta_i) = \xi_i.$$

This is an extension of the map $\pi$ in Theorem \ref{Brauarcentraliser}.

\end{definition} 

\begin{lemma} The map $\Phi$ defined in Definition \ref{cyclbrauerinj} is well defined. \end{lemma}

 Restricting $\Phi$ to the Brauer algebra as a subalgebra of the cyclotomic Brauer gives the well defined map $\pi$ in Theorem \ref{Brauarcentraliser}. Therefore we only need to check the relations of $Br_{k,2}$ involving $\theta_i$. These relations are:  
$$(1) \hspace{1cm}\theta_i^2 = Id,$$
$$(2) \hspace{1cm}t_{i,i+1}\theta_i t_{i,i+1} = \theta_{i+1},$$
$$(3) \hspace{1cm}e_{i-1,i}\theta_i e_{i-1,i} = \delta_1 e_{i-1,i},$$
$$(4) \hspace{1cm}\theta_i \theta_{i+1}  e_{i,i+1} = e_{i,i+1} \theta_i \theta_{i+1}= e_{i,i+1},$$
$$(5) \hspace{1cm}[\theta_i,e_{j,j+1}] = 0 \text{ for } j \neq i,i-1.$$
$$(6) \hspace{1cm}[\theta_i,t_{j,j+1}] = 0 \text{ for } j \neq i,i-1.$$

Since $\xi$ is an involution then relation $(1)$ holds. The second relation follows from the fact that the symmetric group generated by $s_{i,i+1}$ permutes the tensorands of $V^{\otimes k}$, hence it permutes the operators $\xi_i$. 

\begin{remark} It is sufficient to look at $V \otimes V$ for relation (3).
The operators $pr_{12} (\xi_1) \pr_{12}$ is the following map
$$V\otimes V \overset{pr_{12}}{\to} 1 \overset{\xi_1}{\to} V\otimes V \overset{pr_{12}}{\to} 1.$$
Hence $pr_{12}(\xi_1)\pr_{12}$ can be written as $pr_{12}$ followed by an automorpishm of $1$. Automorphisms of $1$ are equivalent to $\mathbb{C}$. To determine $\delta_0$ we just need to work out this automorphism. 
\end{remark}

\begin{lemma} For $G=SO(p,q)$ then $pr_{12} \epsilon_1$ acts on the trivial module $\mathbbm{1} \subset V \otimes V$ in particular:
$$ pr_{12}\xi_1 (v_\mathbbm{1}) = (p-q) v_\mathbbm{1},$$
similarly for $G= Sp_{2n}(\mathbb{R})$,
$$ pr_{12}\xi_1 (v_\mathbbm{1}) = 0  v_\mathbbm{1}.$$
\end{lemma} This lemma shows that relation $(3)$ is mapped to zero by $\Phi$ where 
$$\delta_1 = \begin{cases} p -q & \text{for } G = SO(p,q),\\ 0 & \text{for } G = Sp_{2n}(\mathbb{R}).\end{cases}$$ 
\begin{proof}
Recall that for $SO(p,q)$ (resp. $Sp_{2n}(\mathbb{R})$) the trivial modules is spanned by $v_\mathbbm{1} = \sum_{i=1}^n e_i \otimes e_i$ (resp. $\sum_{i=1}^n f_i \otimes f_i' - f_i'\otimes f_i$). Using the action of $\xi$ on $V$ defined in Remark \ref{xiaction} then 
$\xi_1(v_\mathbbm{1}) = \sum_{i=1}^p e_i \otimes e_i - \sum_{i=p+1}^{p+q} e_i \otimes e_i,$
and $\xi_1(v_\mathbbm{1}) = \sqrt{-1}\sum_{i=1}^n f_i' \otimes f_i' + f_i \otimes f_i.$
The operator $pr_{12}$ is equivalent to $v \mapsto \langle v, v_\mathbbm{1}\rangle v_\mathbbm{1}$. 
For $G=SO(p,q)$ 
$$\langle \xi_1(v_\mathbbm{1}), v_\mathbbm{1}\rangle  =  \langle \sum_{i=1}^p e_i \otimes e_i - \sum_{i=p+1}^{p+q} e_i \otimes e_i, \sum_{i=1}^n e_i \otimes e_i \rangle = p-q.$$
For $G = Sp_{2n}(\mathbb{R})$the vector $\xi_1(v_\mathbbm{1}) = \sqrt{-1}\sum_{i=1}^n f_i' \otimes f_i' + f_i \otimes f_i$ is in the symmetric square $S^2V$ hence has zero inner product with $v_\mathbbm{1}$ which is in the alternating square $\Lambda^2 V$.
\end{proof}

Note that $e_{12}$ acts by the operator $\delta_0 pr_{12}$ which projects on to the trivial $\mathfrak{g}$-submodule of $V\otimes V$. The action of $\xi$ on $V\otimes V$ is $\xi_1\xi_2$, which comes from the $K$ action so commutes with $pr_{12}$, furthermore $\xi_1\xi_2pr_{12}$ is just the action of $\xi$ on the trivial module, hence $\xi_1\xi_2 e_{12} = e_{12}$ and this shows relation $(4)$ maps to zero. Finally relation $(5)$ and $(6)$ map to zero since $[(X)_i,(Y)_j]=0$ for any $i\neq j $.

Hence we have a well defined map $\Phi$ from $Br_{k,2}$ to $\End_\mathfrak{k}(V^{\otimes k})$.  
\end{subsection}
\begin{subsection}{Our map defines an inclusion $Br_{k,2}$ into $\End_\mathfrak{k}(V^{\otimes k})$.}

We need to prove that for $k \leq n$ the map $\Phi$ is injective. The dimension of $Br_{k,m}$ is $\frac{(2k)!m^k}{k!2^k}$, in particular for $Br_{k,2}$ it is $\frac{(2k)!}{k!}$. This is the same size as the number of marked Brauer diagrams. Note that for $m=2$ the set of marked Brauer diagrams and the equivalence class of directed marked Brauer diagrams are the same since reversing the direction does not change the sign because $1 = -1 mod 2$.

\begin{remark} Let $\tilde{d}$ be a marked Brauer diagram of size $k$ with at most $1$ mark on each strand. Let $d$ be the diagram related to $\tilde{d}$ by forgetting the markings. Let $d_1^\theta$ denote the identity diagram with marks corresponding to every marked top horizontal strand in $\tilde{d}$, Let $d_2^\theta$ denote  the identity diagram with marks corresponding to every markes strand that terminates on the bottom row. We make sure that if a mark is on a horizontal strand then that mark is represented on the right most vertex. This decomposition is unique. We can write $\tilde{d}$ as a concatenation of three diagrams 
$d_1^\theta \circ d \circ d_2^\theta.$ We explain this with an example.
\end{remark}

\begin{example}
We take a diagram in $Br_{k,2}$ and write it as an unmarked diagram pre and post multiplied by marked identities.
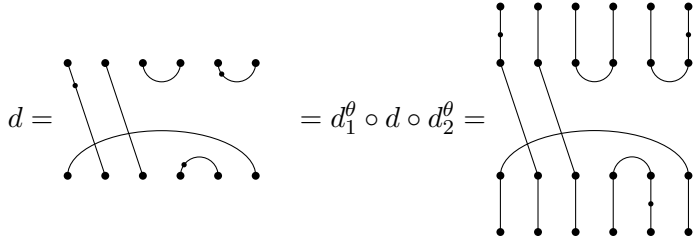
\begin{figure}[ht]
\centerline{
$d=$
\begin{minipage}{30mm}\begin{tikzpicture}[scale=0.5]
  \foreach \x in {0.5,1.5,...,5.5}
    {\fill (\x,3) circle (3pt);
     \fill (\x,0) circle (3pt);}
\begin{scope}[decoration={
    }
    ]
     \draw[postaction={decorate}] (1.5,3) -- (2.5,0); 
    \draw[postaction={decorate}] (2.5,3) arc (180:360:0.5 and 0.5);
     \draw[postaction={decorate}] (0.5,0) arc (180:360:2.5 and -1.2);
  \end{scope}    
\begin{scope}[decoration={
    markings,
    mark= at position 0.2 with{\fill circle (1pt);}
    }
    ] 
    \draw[postaction={decorate}]  (0.5,3) -- (1.5,0);    
    \draw[postaction={decorate}] (3.5,0) arc (180:360:0.5 and -0.5);    
    \draw[postaction={decorate}] (4.5,3) arc (180:360:0.5 and 0.5);            
  \end{scope}
\end{tikzpicture}\end{minipage}
$=d_1^\theta \circ d \circ d_2^\theta=$
\begin{minipage}{45mm}\begin{tikzpicture}[scale=0.5]
  \foreach \x in {0.5,1.5,...,5.5}
    {\fill (\x,7.5) circle (3pt);
     \fill (\x,6) circle (3pt);
     \fill (\x,3) circle (3pt);
     \fill (\x,1.5) circle (3pt);}
\begin{scope}[decoration={    }
    ]
     \draw[postaction={decorate}] (0.5,3) -- (0.5,1.5);
     \draw[postaction={decorate}] (1.5,3) -- (1.5,1.5);
     \draw[postaction={decorate}] (2.5,3) -- (2.5,1.5);
     \draw[postaction={decorate}] (3.5,3) -- (3.5,1.5);
     \draw[postaction={decorate}] (5.5,3) -- (5.5,1.5); 
     \draw[postaction={decorate}] (1.5,7.5) -- (1.5,6);
     \draw[postaction={decorate}] (2.5,7.5) -- (2.5,6);
     \draw[postaction={decorate}] (3.5,7.5) -- (3.5,6);
     \draw[postaction={decorate}] (4.5,7.5) -- (4.5,6); 
     \draw[postaction={decorate}]  (0.5,6) -- (1.5,3); 
      \draw[postaction={decorate}] (1.5,6) -- (2.5,3); 
     \draw[postaction={decorate}] (3.5,3) arc (180:360:0.5 and -0.5);    
    \draw[postaction={decorate}] (4.5,6) arc (180:360:0.5 and 0.5);
    \draw[postaction={decorate}] (2.5,6) arc (180:360:0.5 and 0.5);
     \draw[postaction={decorate}] (0.5,3) arc (180:360:2.5 and -1.2);
  \end{scope}      
\begin{scope}[decoration={
    markings,
    mark= at position 0.5 with{\fill circle (1pt);}
    }
    ]
     \draw[postaction={decorate}] (4.5,3) -- (4.5,1.5);   
     \draw[postaction={decorate}] (5.5,7.5) -- (5.5,6);  
     \draw[postaction={decorate}]  (0.5,7.5) -- (0.5,6);     
  \end{scope}
\end{tikzpicture}\end{minipage}}
 \caption{Writing a marked diagram as an unmarked diagram post and pre multiplied by marked identities}
\label{xyelts4}
\end{figure}
\end{example}
It is easy to write $d_1^\theta$ in terms of the generators, $t_{i,i+1},e_{i,i+1}$ and $\theta_i$. Suppose $d_1^\theta$ has marks on a subset $I \subset \{1,...,n\}$ then $d_1^\theta = \prod_{i \in I}\theta_i$.
The map $\Phi$ takes $\tilde{d}$ to $\Phi(d_1^\theta)\pi(d)\Phi(d_2^\theta)= \prod_{i\in I_1}\theta_i \pi(d) \prod_{i\in I_2}\theta_2$.

\begin{lemma} The map $\Phi$ is injective for $k \leq n$. \end{lemma} 

\begin{proof} Suppose there are two diagrams $\tilde{d}$ and $\tilde{f}$ that get mapped to the same element by $\Phi$. That is $\Phi(\tilde{d}) = \Phi (\tilde{f})$.
This means that $$\Phi(d_1^\theta)\pi(d)\Phi(d_2^\theta) = \Phi(f_1^\theta)\pi(f)\Phi(f_2^\theta).$$ 
By considering the projection from $Br_{k,2}$ to $Br_k$ that takes $\theta_i$ to $Id$ this shows that 
$\pi(d) = \pi(f)$ and since $\pi$ is injective then $d =f$.
So we have reduced it to the case that $\tilde{d}$ and $\tilde{f}$ are the same unmarked diagram. 
So $\tilde{d}$ and $\tilde{f}$ have different markings on the same diagram.  

The operator $\xi_i$ does not commute with the operator $e_{i,i+1}$, furthermore there is no way in our generating set to write $\xi_i e_{i,i+1} = e_{i,i+1}\xi_i +w$ where $w$ is a word in the generators.  Therefore $\tilde{d}$ and $\tilde{f}$ can only differ by marking on through strands. Every marking on a through strand gives a unique value in $d_1^\theta$ or $f_1^\theta$. 
By acting by the symmetric product on $\tilde{d}$ and $\tilde{f}$ we can assume that they differ by markings on two vertical strands. This implies that for some $i \neq j$ then $\xi_i = \xi_j$ which isn't true so there does not exist two different marked diagrams which map to the same operator under $\Phi$. 
\end{proof}
\end{subsection}
\end{section}

\begin{section}{Proving the full centraliser result}\label{surjection}

In the previous section we proved that the cyclotomic Brauer algebra is a subalgebra of $\End_\mathfrak{k}(V^{\otimes k})$ in this section we show that the cyclotomic Brauer is the entire centraliser of $\mathfrak{k}$. We split this section into two, one for the symplectic group and another for the orthogonal groups. 

\begin{subsection}{The symplectic group}\label{sphomspaces}

In this subsection $G$ will be fixed to $Sp_{2n}(\mathbb{R})$, $\mathfrak{g}$ is $\mathfrak{sp}_{2n}$, $\mathfrak{k}$ is $\mathfrak{gl}_n$. The standard module of $Sp_{2n}(\mathbb{R})$is $V = \mathbb{C}^{2n}$ the standard module of $\mathfrak{gl}_n$ is denoted $V_\mathfrak{gl} = \mathbb{C}^n$  with dual module $V_\mathfrak{gl}^*$. We will denote the standard basis by $e_i$ which should not lead to any confusion.

\begin{definition} A walled Brauer diagram of rank $s,t$ is a Brauer diagram from $s +t$ vertices to $s+t$ vertices with a wall between the first $s$ vertices and the last $t$ vertices, such that every through strand does not cross the wall and every horizontal strand does cross the wall. \end{definition}

\begin{figure}[ht]
\centerline{
$x=$
\begin{minipage}{45mm}\begin{tikzpicture}[scale=0.5]
  \foreach \x in {0.5,1.5,...,5.5}
    {\fill (\x,3) circle (3pt);
     \fill (\x,0) circle (3pt);}
\begin{scope}[
    ]
     \draw[postaction={decorate}] (1.5,3) -- (2.5,0); 
    \draw[postaction={decorate}] (2.5,3) arc (180:360:0.5 and 0.5);
     \draw[postaction={decorate}] (0.5,0) arc (180:360:2.5 and -1.2);
     \draw[dotted] (3,3.5) -- (3,-.5);
  \end{scope}    
\begin{scope}
    \draw[postaction={decorate}] (3.5,0) -- (5.5,3);    
    \draw[postaction={decorate}] (4.5,3) --(4.5,0);            
  \end{scope}
\begin{scope}
     \draw[postaction={decorate}]  (0.5,3) -- (1.5,0);    
  \end{scope}
\end{tikzpicture}\end{minipage}
$y=$\begin{minipage}{45mm}\begin{tikzpicture}[scale=0.5]
  \foreach \x in {0.5,1.5,...,5.5}
    {\fill (\x,3) circle (3pt);
     \fill (\x,0) circle (3pt);}
\begin{scope}
     \draw[postaction={decorate}] (1.5,3) -- (0.5,0);
      \draw[dotted] (3,3.5) -- (3,-.5);
  \end{scope}   
\begin{scope}
     \draw[postaction={decorate}] (2.5,0) arc (180:360:1.5 and -1.2);
  \end{scope}    
\begin{scope}
    \draw[postaction={decorate}] (3.5,0) --(4.5,3);  
    \draw[postaction={decorate}] (2.5,3) arc (180:360:1.5 and 1.2);  
    \draw[postaction={decorate}] (4.5,0) -- (3.5,3);            
  \end{scope}
\begin{scope}
     \draw[postaction={decorate}]  (0.5,3) -- (1.5,0);    
  \end{scope}
\end{tikzpicture}\end{minipage}}
 \caption{Two walled Brauer diagrams in $WB_{3,3}(\delta)$}
\end{figure}
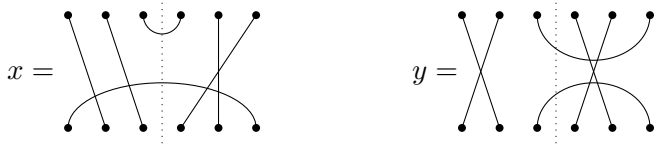

\begin{definition} The walled Brauer algebra associated to integers $s,t$ and a variable $\delta$ is the subalgebra of the Brauer algebra $Br_{s+t}[\delta]$ generated by walled Brauer diagrams.\end{definition}

\begin{theorem}\cite{BGCMH94}\label{WalledBrauercentraliser} Let $\rho$ denote the action of $GL_n$ (resp. $\mathfrak{gl}_n$) on $V_\mathfrak{gl}$. Define $\pi$ from $WBr_{s,t}[n]$ to $\End V^{\otimes s}\otimes (V_\mathfrak{gl}^*)^{\otimes t}$ by:
$$\pi(t_{i,i+1}) = s_{i,i+1} \text{ for all } i \neq s,$$
$$\pi(e_{s,s+1}) =   n pr_{s,s+1}, $$
where $pr_{s,s+1}$ is the projection of the $s^{th}$ and $s+1^{st}$ factor $V_\mathfrak{gl}\otimes V_\mathfrak{gl}^*$ to the trivial submodule spanned by $\sum_{i=1}^n e_i \otimes e_i^*$.

We have the following equalities: 

$$\begin{array}{lcr}
\rho^s \otimes (\rho^*)^t(\mathfrak{gl}_n) & = & \End_{Br_{s,t}} (V_\mathfrak{gl}^{\otimes s}\otimes (V_\mathfrak{gl}^*)^{\otimes t}),\\
\pi(Br_{s,t}[n]) & = & \End_\mathfrak{gl} (V_\mathfrak{gl}^{\otimes s}\otimes (V_\mathfrak{gl}^*)^{\otimes t}).\\
\end{array} 
$$
Furthermore if $s+t \leq n$ then the representation $\pi$ is faithful. \end{theorem}

The dimension of the Walled brauer algebra $WB_{s,t}$ is $(s+t)!$. To show this we can define a vector space isomorphism between $WB_{s,t}$ and $\mathbb{C}[S_{s+t}]$ by flipping the first $s$ top vertices with the first $s$ bottom vertices. Here we think of $\mathbb{C}[S_{s+t}]$ as the diagram algebra of diagrams that do not have any horizontal strands.

\begin{remark} The standard representation $V$ of $Sp_{2n}(\mathbb{R})$ decomposes as $V_\mathfrak{gl} \oplus V_\mathfrak{gl}^*$ as a $\mathfrak{gl}_n$ module. This allows us to further study $\End_{\mathfrak{gl}_n}( V^{\otimes k})$\end{remark}
The above remark can be extended to $V^{\otimes k}$. It shows that as $\mathfrak{gl}_n$ modules 
$$V^{\otimes k} = \bigoplus_{s=0}^k {k \choose s} V_\mathfrak{gl}^{\otimes s} \otimes (V_\mathfrak{gl}^*)^{\otimes n-s}.$$

\begin{remark}\label{homspacezero} The space $\Hom_\mathfrak{gl}( V_\mathfrak{gl}^{\otimes s} \otimes (V_\mathfrak{gl}^*)^{\otimes n-s},  V_\mathfrak{gl}^{\otimes t} \otimes (V_\mathfrak{gl}^*)^{\otimes n-t})$ is zero for $s \neq t$.\end{remark}

\begin{lemma}\label{endsplitting} Let $V$ be the standard representation of $Sp_{2n}(\mathbb{R})$and let $V_\mathfrak{gl}$ be the standard representation of $\mathfrak{gl}_n$ then 
$$\End_\mathfrak{gl}(V^{\otimes k}) = \bigoplus_{s=0}^k\End (\mathbb{C}^{k\choose s} )Br_{s,k-s}[n].$$
\end{lemma}

\begin{proof} We use the decomposition 
$$V^{\otimes k} = \bigoplus_{s=0}^k {k \choose s} V_\mathfrak{gl}^{\otimes s} \otimes (V_\mathfrak{gl}^*)^{\otimes n-s}$$
to write 
$$\End_\mathfrak{gl}(V^{\otimes k}) = \bigoplus_{s=0, t= 0}^k\Hom (\mathbb{C}^{k\choose s}, \mathbb{C}^{k \choose t} )\otimes \Hom_\mathfrak{gl}(V_\mathfrak{gl}^{\otimes s} \otimes( V_\mathfrak{gl}^*)^{\otimes k-s},V_\mathfrak{gl}^{\otimes t} \otimes( V_\mathfrak{gl}^*)^{\otimes k-t}).$$

However Remark \ref{homspacezero} states that $\Hom_\mathfrak{gl}( V_\mathfrak{gl}^{\otimes s} \otimes (V_\mathfrak{gl}^*)^{\otimes n-s},  V_\mathfrak{gl}^{\otimes t} \otimes (V_\mathfrak{gl}^*)^{\otimes n-t})$ is zero if $s \neq t$. Therefore all the cross terms cancel out and we end up with  

$$\End_\mathfrak{gl}(V^{\otimes k}) = \bigoplus_{s=0}^k\End (\mathbb{C}^{k\choose s} )\End_\mathfrak{gl}(V_\mathfrak{gl}^{\otimes s} \otimes( V_\mathfrak{gl}^*)^{\otimes k-s}).$$

Then using Theorem \ref{WalledBrauercentraliser} to replace $\End_\mathfrak{gl}(V_\mathfrak{gl}^{\otimes s} \otimes( V_\mathfrak{gl}^*)^{\otimes k-s})$ with $Br_{s,k-s}$ gives the result.
\end{proof}

\begin{lemma}\label{dimofend}Let $V$ be the standard module for $Sp_{2n}(\mathbb{R})$, the dimension of $\End_\mathfrak{gl} (V^{\otimes k})$ is $\frac{(2k)!}{k!}$ \end{lemma} 

\begin{proof} Using Lemma \ref{endsplitting} we have that the dimension of 
$\End_\mathfrak{gl}(V^{\otimes k}$ is 
$$\sum_{s=0}^k {k \choose s}^2 \Dim (Br_{k-s,s})$$
$$= \sum_{s=0}^k {k \choose s}^2 (k-s+s)!$$
then using the fact that $\sum_{i=0}^k {k \choose i}^2 = {2k \choose k}$ we obtain 
$$= k! {2k \choose k} = \frac{ (2k)!}{k!}.$$
\end{proof}

\begin{corollary} The map $\Phi$ from $Br_{k,2}[-n,0]$ to $\End_{\mathfrak{gl}} (V^{\otimes k})$ is an isomorphism for $k \leq n$. \end{corollary}

\begin{proof} We have already shown that $\Phi$ is injective. Also the dimension of $Br_{k,2}[-n,0]$ is $\frac{(2k)!}{k!}$ which by Lemma \ref{dimofend} is equal to the dimension of $\End_{\mathfrak{gl}}(V^{\otimes k})$.
\end{proof}

\begin{theorem} Let $G = Sp_{2n}(\mathbb{R})$with maximal compact subgroup $K = U(n)$, and $\mathfrak{k} = Lie(K)$. When $k \leq n$ there is a natural isomorphism between $\End_K(V^{\otimes k})$ and the cyclotomic Brauer algebra $Br_{k,2}[-n,0]$ \end{theorem}

\begin{proof} The map $\Phi$ is an isomorphism.\end{proof}
\end{subsection}

\begin{subsection}{The orthogonal groups}
 In this section we fix $G = SO(p,q)$. $\mathfrak{g}_0= \mathfrak{so}(p,q)$, $\mathfrak{g}= \mathfrak{so}_{p+q}$ with $p+q$ odd, we fix a maximal compact subgroup of $SO(p,q)$ to be $SO(p) \times SO(q)$ diagonaly embedded in to $G$, $\mathfrak{k} = \mathfrak{so}_p \oplus \mathfrak{so}_q$. We fix $V=\mathbb{C}^{p+q}$ to be the standard representation of $SO(p,q)$, similarly $V_p= \mathbb{C}^p$ and $V_q= \mathbb{C}^q$ are standard representations for $SO(p)$ and $SO(q)$. We abuse notation and write the $SO(p) \times SO(q)$ module $V_p \otimes \mathbbm{1}$ as $V_p$, and do the same for $\mathbbm{1}\otimes V_q$ as $V_q$.
 
 \begin{definition} An uneven Brauer diagram from $s$ vertices to $t$ vertices is a pairing of  $s+t$ vertices, we arrange the diagram to have $s$ top vertices and $t$ bottom vertices. If $s+t$ is even there are $\Dim (Br_{\frac{s+t}{2}})$ diagrams, if $s+t$ is odd then there are none.
 \end{definition}
 The above uneven Brauer diagrams are similar to the dangles used in \cite{BCV12}.
 
 \begin{example} Examples of Brauer diagrams from $6$ to $4$ vertices and $5$ to $7$

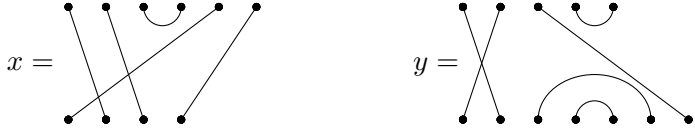
\begin{figure}[ht]
\centerline{
$x=$
\begin{minipage}{45mm}\begin{tikzpicture}[scale=0.5]
  \foreach \x in {0.5,1.5,...,5.5}
  \foreach \y in {0.5,1.5,...,3.5}
    {\fill (\x,3) circle (3pt);
     \fill (\y,0) circle (3pt);}
\begin{scope}
     \draw[postaction={decorate}] (3.5,0) -- (5.5,3);    
     \draw[postaction={decorate}] (1.5,3) -- (2.5,0); 
       \draw[postaction={decorate}]  (0.5,3) -- (1.5,0);    
       \draw[postaction={decorate}] (2.5,3) arc (180:360:0.5 and 0.5);
     \draw[postaction={decorate}] (0.5,0) -- (4.5,3);
  \end{scope}    
\end{tikzpicture}\end{minipage}
$y=$\begin{minipage}{45mm}\begin{tikzpicture}[scale=0.5]
\foreach \x in {0.5,1.5,...,4.5}
\foreach \y in {0.5,1.5,...,6.5}  
    {\fill (\x,3) circle (3pt);
     \fill (\y,0) circle (3pt);}
\begin{scope}
     \draw[postaction={decorate}] (1.5,3) -- (0.5,0);
  \end{scope}   
\begin{scope}
   \draw[postaction={decorate}] (2.5,3) -- (6.5,0);  
       \draw[postaction={decorate}]  (0.5,3) -- (1.5,0);      
    \draw[postaction={decorate}] (3.5,0) arc (180:360:0.5 and -0.5);  t
    \draw[postaction={decorate}] (3.5,3) arc (180:360:0.5 and 0.5);   
     \draw[postaction={decorate}] (2.5,0) arc (180:360:1.5 and -1.2);
  \end{scope}    
\begin{scope}[decoration={
    markings,
    mark= at position 0.2 with{\fill circle (1pt);},
    mark=at position 0.5 with {\arrow{>}}
    }
    ]       
  \end{scope}
\begin{scope}[decoration={
    markings,
    mark= at position 0.2 with{\fill circle (1pt);},
    mark=at position 0.7 with {\arrow{>}},
    mark= at position 0.35 with{\fill circle (1pt);}}
    ]
  \end{scope}
\end{tikzpicture}\end{minipage}}
 \caption{Two uneven Brauer diagrams}
\label{xyelts5}
\end{figure}

\end{example}

 \begin{lemma}\label{homspacebrauer} For $p\geq \max(s,t)$ The Hom space $\Hom_{\mathfrak{sp}_p}(V_p^{\otimes s}, V_p^{\otimes t})$ is in correspondence with vector space spanned by Brauer diagrams from $s$ vertices to $t$ vertices.

\end{lemma}

\begin{proof}
Wlog let $s \geq t$
Look at diagrams inside $Br_s$. There is a surjection onto $\Hom_{\mathfrak{sp}_p}(V_p^{\otimes s}, V_p^{\otimes t})$ which is independent of pairings of the bottom $s-t$ vertices. This is done by fixing $V_p^{\otimes t}$ as a subspace of $V^{\otimes s}$.
\end{proof}

\begin{remark} The standard representation $V$ of $SO(p,q)$ decomposes as $V_p \oplus V_q$ as a $\mathfrak{so}_p \oplus \mathfrak{so}_q$ module. 
The above remark can be extended to $V^{\otimes k}$ it shows that as  $\mathfrak{so}_p \oplus \mathfrak{so}_q$ modules 
$$V^{\otimes k} = \bigoplus_{s=0}^k {k \choose s} V_p^{\otimes s} \otimes V_q^{\otimes n-s}.$$
\end{remark}

\begin{lemma}\label{SOendsplitting} Let $V$ be the standard representation of $SO(p,q)$ then 
$$\End_\mathfrak{k}(V^{\otimes k}) =  \bigoplus_{s=0, t= 0}^k\Hom (\mathbb{C}^{k\choose s}, \mathbb{C}^{k \choose t} )\otimes \Hom_{\mathfrak{so}_p}(V_p^{\otimes s},V_p^{\otimes t}) \otimes \Hom_{\mathfrak{so}_q}( V_q^{\otimes k-s},V_q^{\otimes k-t}).$$
\end{lemma}

\begin{proof} We use the decomposition 

$$V^{\otimes k} = \bigoplus_{s=0}^k {k \choose s} V_p^{\otimes s} \otimes V_q^{\otimes n-s}.$$
to write $$\End_\mathfrak{k}(V^{\otimes k}) =  \bigoplus_{s=0, t= 0}^k\Hom (\mathbb{C}^{k\choose s}, \mathbb{C}^{k \choose t} )\otimes \Hom_{\mathfrak{so}_p}(V_p^{\otimes s},V_p^{\otimes t}) \otimes \Hom_{\mathfrak{so}_q}( V_q^{\otimes k-s},V_q^{\otimes k-t}).$$

\end{proof}

\begin{lemma} Let $V$ be the standard mofule for $SO(p,q)$, the dimension of $\End_\mathfrak{k} (V^{\otimes k})$ is $\frac{(2k)!}{k!}$ \end{lemma} 

\begin{proof} Using Lemma \ref{SOendsplitting} we have that the dimension of 
$\End_\mathfrak{k}(V^{\otimes k})$ is 
$$\sum_{s,t=0}^k {k \choose s}{t \choose t} \Dim (\Hom_{\mathfrak{so}_p}(V_p^{\otimes s},V_p^{\otimes t})) \times \Dim( \Hom_{\mathfrak{so}_q}( V_q^{\otimes k-s},V_q^{\otimes k-t})))$$
then using Lemma \ref{homspacebrauer} the fact that $\Hom$ spaces have the same dimension as Brauer algebras we obtain
$$\sum_{s,t=0}^k {k \choose s}{t \choose t} \Dim (Br_{\frac{s+t}{2}}) \times \Dim(Br_{k - \frac{s+t}{2}})$$
$$= \sum_{s=0}^k \sum_{t=0, s+t \text{ even} }^{0} {k \choose s}{k \choose t}\frac{(s+t)!}
{\frac{(s+t)}{2}!2^{\frac{s+t}{2}}}
\frac{(2k-(s +t))!}{(k-\frac{s+t}{2})!2^{k-\frac{s+t}{2}}}.$$

$$= \sum_{s=0}^k \sum_{t=0, s+t \text{ even} }^{k} {k \choose s}{k \choose t} \frac{{k \choose \frac{s+t}{2}}}{{2k \choose s+t}} \frac{k!}{(2k)!2^k}.$$
Rearranging the sum we find

$$=    \frac{k!}{(2k)!2^k}\sum_{r=0 =\frac{s+t}{2}}^{k}\frac{{k \choose r}}{{2k \choose 2r}} \sum_{s=0}^k{k \choose s}{k \choose 2r-s}.$$

Then using $\sum_{s=0}^{k} {k \choose s}{k \choose m-s} = {2k \choose m}$

$$=    \frac{k!}{(2k)!2^k}\sum_{r=0}^{k}\frac{{k \choose r}}{{2k \choose 2r}} {2k \choose 2r}.$$

$$=    \frac{k!}{(2k)!2^k}\sum_{r=0}^{k}\frac{{k \choose r}}{{2k \choose 2r}} {2k \choose 2r}.$$

$$=    \frac{k!}{(2k)!2^k}\sum_{r=0}^{k} {k \choose r}.$$
Using $\sum_{r=0}^{k} {k \choose r} =2^k$, we arrive at $\frac{k!}{(2k)}$ as desired.
\end{proof}

\begin{corollary} The map $\Phi$ from $Br_{k,2}[n,p-q]$ to $\End_{\mathfrak{so}_p \oplus \mathfrak{so}_q} (V^{\otimes k})$ is an isomorphism for $k \leq p+q$, with $p+q$ odd. \end{corollary}

\end{subsection}
\end{section}

\begin{section}{Consequences}\label{conseq}
For $G = Sp_{2n}(\mathbb{R})$, and $k \leq n$ the algebra $\End_K(V^{\otimes k})$ is isomorphic to the cyclotomic Brauer algebra $Br_{k,2}$ with parameters $\delta_0= -n$ and $\delta_1 = 0$. Furthermore our workings from Section \ref{sphomspaces} show that the cyclotomic Brauer algebra $Br_{k,2}[-n,0]$ is isomorphic to a direct sum of matrix algebra tensored with walled brauer algebras.

From Brundan and Stroppel's work on the gradings of Walled Brauer algebras and Khovanov's arc algebra \cite{BS12} we can consider the cyclotomic Brauer algebra at parameters $-n$ and $0$ to be graded. This then endows the cyclotomic Brauer algebra with a grading. It would be interesting to study the graded representation theory of the cyclotomic Brauer algebra with these parameters. 

\begin{theorem} For $\delta_0=-n$ and $\delta_1= 0$ the cyclotomic Brauer algebra associated to the hyperoctahedral group $Br_{k,2}$ is isomorphic to: 

$$Br_{k,2}[-n,0] \cong \bigoplus_{s=0}^kM_{{k\choose s}\times {k \choose s}} WBr_{s, k-s}[n].$$

\end{theorem}

It is interesting that in this particular case the cyclotomic algebra breaks into smaller walled Brauer algebras so cleanly. This opens up interesting question regarding Morita equivalences. 

\begin{corollary} The cyclotomic Brauer algebra $Br_{k,2}[-n,0]$ is Morita equivalent to a direct sum of walled Brauer algebras. 
$$Br_{k,2}[-n,0] - rep - \mapsto \bigoplus_{s=0}^{k} WBr_{k,k-s}[n]-rep$$
\end{corollary} 

Due to \cite{BCV12}, the irreducible representations of $BR_{k,2}$ over $\mathbb{C}$ are in one-to-one correspondence with bipartitions $\underline{\lambda} = (\lambda_1,\lambda_2)$ of $k-2l$ for $l = 0,...,\lfloor \frac{k}{2} \rfloor$. In \cite{CVDM08} they show that the irreducible representation of $WB_{r,s}$ are in one-to-one corresondence with bipartitions $\underline{\lambda} = (\lambda_1,\lambda_2)$ such that $\lambda_1 \vdash r -l$, $\lambda_2 \vdash s-l$ for $l = 0,...,min(r,s)$. 

\begin{remark} 
The set of bipartions of $(\lambda_1,\lambda_2)$ $k -2l$ for a fixed $k$ and $l = 0,...,\lfloor \frac{k}{2} \rfloor$ is equal to the union of bipartitions $\underline{\lambda} = (\lambda_1,\lambda_2)$ such that $\lambda_1 \vdash r -l$, $\lambda_2 \vdash s-l$ for $l = 0,...,min(r,s)$ and $r+s = k$.
It seems sensible that the Morita equivalence above would sent a representation of $Br_{k,2}$ associated to $(\lambda_1,\lambda_2)$ to the representation of $WB_{r,s}$ corresponding to the same bipartition. 
\end{remark} 

Bowman, Cox and de Visscher \cite{BCV12} showed that the cyclotomic Brauer algebra has underlying combinatorics related to Brauer algebras and Walled Brauer algebras but did not show a Morita equivalence for all parameters. This Corollary gives an actual Morita equivalence, although only for particular parameters. It would be an interesting question so see if we can utilise this isomorphism to use ideas from the walled Brauer algebras to study the cyclotomic Brauer algebra or vice versa. For example Brundan and Stroppel \cite{BS12} proved that the walled Brauer algebra has a grading and is Morita equivalent to a truncation of the Khovanov-arc algebra. 

We can also utilise the standard arguments from Schur-Weyl duality. That is of the double centraliser result: 
\begin{definition} For a Lie group $G$ which is a subgroup of $GL(V)$, we say the polynomial representations are those representations that occur as a summand of $V^{\otimes k}$ for some $k\in \mathbb{N}$, and the $r^{th}$ polynomial representations are the ones occuring on $V^{\otimes r}$
\end{definition} 
For a given $(\mathfrak{g},K)$-module $X$, it's $K$-types are the direct sum $K$-submodules that occur inside $X$ restricted to $K$.  
\begin{theorem} For $G = O(p,q)$ or $Sp_{2n}(\mathbb{R})$ the $K$-types arising in the $r^{th}$ polynomial representations are in one-to-one correspondence with irreducible representation of the cyclotomic Brauer algebra $Br_{k,2}[\delta_0,\delta_1]$. \end{theorem}

\begin{remark} This Theorem then labels every $K$-type of $V^{\otimes k}$ with a bipartition $\underline (\lambda) = (\lambda_1,\lambda_2)$ of $k -2l$ for an $l =0,...,\lfloor \frac{k}{2} \rfloor$. The $\mathfrak{k}$ types are either $\mathfrak{gl}_n$ modules or $\mathfrak{so}(p)\oplus \mathfrak{so}(q)$ modules which have a rich structure of combinatorial description with respect to highest weight theory. Does this combinatorial structure agree in a nice succinct way?\end{remark}

We intend to release a further paper studying the space $\End_K(X \otimes V^{\otimes k})$ which will define a hyperoctahedral equivalent to the Nazarov-Wenzl algebra \cite{N90}. It will generalise Arakawa-Suzuki functors \cite{AS98} to $(\mathfrak{g},K)$-modules following the work of Ciubotaru and Trapa \cite{CT11}.

\end{section}

\bibliography{/home/k/google_drive/Dphil_work/bib.bib}        
\bibliographystyle{abbrv}  

\end{document}